\documentclass[11pt]{amsart}

\usepackage{amsmath}
\usepackage[font=small]{caption}
\usepackage{amsthm}
\usepackage{amsfonts}
\usepackage[pdftex]{graphicx}
\usepackage[margin=.5cm, format=hang]{subfig}
\usepackage[alphabetic]{amsrefs}
\usepackage{amssymb}
\usepackage{tikz-cd}
\usepackage{pinlabel}
\usepackage{morefloats}
\usepackage{enumitem}
\usepackage{fullpage}
\usepackage{array}
\usepackage{multirow}
\usepackage{sidecap}
\usepackage{hhline}
\usepackage{multirow}
\usepackage{blkarray}
\usepackage{xcolor}
\usepackage{comment}
\usepackage{subfig}

\usepackage{listings}

\usepackage{accents}
\usepackage{url}
\usepackage{hyperref}
\hypersetup{
    colorlinks=true,
    linkcolor= violet,
    citecolor= violet,      
    urlcolor= violet,
    pdftitle={A Dunfield--Gong $4$-Sphere is Standard},
    pdfpagemode=FullScreen
}

\newcommand{\Z}{\mathbb{Z}}

\newcommand{\vphi}{\varphi}

\newtheorem{theorem}{Theorem}[section]
\newtheorem{corollary}{Corollary}[theorem]
\newtheorem{lemma}[theorem]{Lemma}

\newtheorem{conjecture}{Conjecture}[section]

\theoremstyle{definition}
\newtheorem{definition}{Definition}[section]
\newtheorem{remark}{Remark}[section]
\newtheorem{question}{Question}[section]

\newtheoremstyle{named}{}{}{\itshape}{}{\bfseries}{.}{.5em}{\thmnote{#3}}
\theoremstyle{named}
\newtheorem*{namedtheorem}{Theorem}

\usepackage{graphicx} 
\graphicspath{{images/}}

\begin{document}
\title[A Dunfield--Gong $4$-Sphere is Standard]{A Dunfield--Gong $4$-Sphere is Standard} 
\author{Trevor Oliveira-Smith}
\address{\hskip-\parindent
Trevor Oliveira-Smith\\
Department of Mathematics \\
University of California, Davis\\
Davis, CA 95616, USA
}
\email{tdoliveirasmith@ucdavis.edu}
\date{\today}

\begin{abstract}
In this paper, we standardize a homotopy $4$-sphere constructed by Dunfield and Gong. As a corollary, we show that the $18$-crossing knot $18_{\text{nh}00000601}$, which is not known to be ribbon, is slice in the standard $4$-ball. Thus, $18_{\text{nh}00000601}$ serves as a potential counterexample to the Slice-Ribbon Conjecture. In addition, we show that the same knot bounds a fibered handle-ribbon disk in $B^{4}$.
\end{abstract}

\maketitle

\section{Introduction}\label{sec: intro}
Two long standing problems in modern low-dimensional topology are the \textit{Slice-Ribbon Conjecture} and the \textit{Smooth $4$-dimensional Poincar\'{e} Conjecture}. We describe both conjectures briefly, and how approaches to the first can relate to the second. \newline

A knot $K\subset S^{3}$ is \textit{slice} in the standard $4$-ball $B^{4}$ if it bounds a smooth properly embedded disk $D\subset B^{4}$ such that $\partial (B^{4},D) = (S^{3},K)$. The disk $D$ is called a \textit{slice disk} for $K$. A slice disk $D$ for $K$ is called \textit{ribbon} if $D$ can be isotoped (rel.\ boundary) so that the restriction of the height function on $B^{4} = B^{3}\times I\to I$ to $D$ has no index $2$ critical points. A knot $K$ in $S^{3}$ is called \textit{ribbon} if it bounds a ribbon disk in $B^{4}$. The Slice-Ribbon Conjecture, originally stated in \cite{Fox}, posits that every slice knot must be ribbon. There are two primary strategies one can use to determine whether or not a given knot is slice. The first strategy is to obstruct sliceness. There are classical invariants which can obstruct \textit{topological sliceness} (i.e.\ $K$ bounds a \textit{topological disk} in $B^{4}$). These classical invariants include the \textit{signature} of $K$, where $\sigma(K) = 0$ if $K$ is topologically slice. Another obstruction to topological sliceness comes from the Fox--Milnor condition of \cite{FoxMil}, which states that the Alexander polynomial of a topologically slice knot must be of the form $f(t)f(t^{-1})$ for some Laurent polynomial $f(t)$. Obstructions to smooth sliceness arise from knot Floer homology \cite{OzSz} and Khovanov homology \cite{Ras}. Conversely, one can show that a given knot $K$ in $S^{3}$ is slice either by explicitly constructing a ribbon disk for $K$ (by finding bands which surger $K$ to an unlink) or by constructing a specific $0$-trace embedding of $K$ into $S^{4}$ as in the classical Trace Embedding Lemma (see Lemma \ref{lem:traceemblem} for more details). \newline

By work of Freedman in \cite{Fr}, it is known that any homotopy $S^{4}$ is \textit{homeomorphic} to $S^{4}$. The Smooth $4$-dimensional Poincar\'{e} Conjecture posits that any smooth $4$-manifold $X$ which is homotopy-equivalent to $S^{4}$ must actually be \textit{diffeomorphic} to $S^{4}$. If there exists a homotopy $4$-sphere $X$ not diffeomorphic to $S^{4}$, such an $X$ would be called an \textit{exotic $S^{4}$}. Despite many examples of exotic $4$-manifolds \cite{Fr,Taubes,AP}, numerous attempts to construct an exotic $S^{4}$ (\cite{CS, Calegari,MP}) have so far been unsuccessful (\cite{Ak_2010,Gompf_2010,Iwaki2025,MeZ,cha2024, Nak}). One of the hurdles in trying to prove (or find a counterexample to) the Smooth Poincar\'{e} Conjecture is that many invariants known to distinguish smooth structures on $4$-manifolds vanish on homotopy $4$-spheres. A modern strategy posed by Manolescu and Piccirillo in \cite{MP} tries to circumvent this problem. We describe their strategy briefly.\newline

The Manolescu--Piccirillo strategy is to seek a pair of knots $K,K'\subset S^{3}$ with homeomorphic $0$-surgeries where $K'$ is slice through a smooth disk $D\subset B^{4}$, but $K$ is not slice. One could then take the disk complement $B^{4}\setminus\eta(D)$ and glue it to the $0$-trace of $K$. The result would be a homotopy $4$-sphere $X$ where $K$ is slice in $X\setminus\eta(B^{4})$ but not in the standard $4$-ball. This could only occur if $X$ is not diffeomorphic to $S^{4}$.  \newline

Recently in \cite{DG}, Dunfield and Gong used computational methods to find pairs of knots on which to apply the Manolescu--Piccirillo strategy. Dunfield and Gong constructed a pair of knots $K_{G}$ and $K_{B}$ sharing the same zero surgery, where $K_{B}$ is ribbon and the knot $K_{G}$ (known as $18_{\text{nh}00000601}$ in \cite{DG}) is not obviously slice. We depict these knots in Figure \ref{fig:K18andK31}.

\begin{figure}[h]
  \centering
  \subfloat[$K_{G}$]{\includegraphics[scale = 0.8]{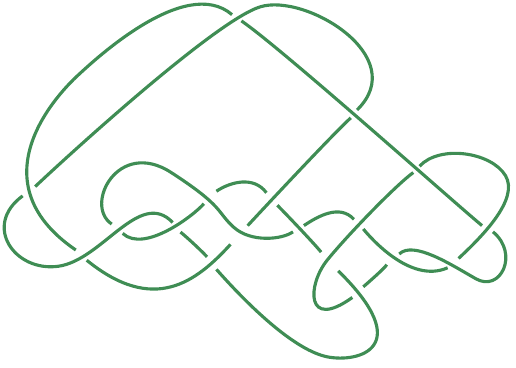}\label{fig:K18}}
  \hfill
  \subfloat[$K_{B}$]{\includegraphics[scale = 0.75]{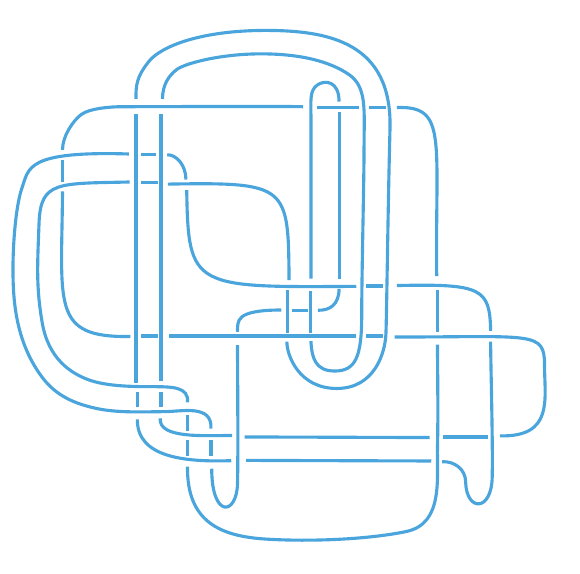}\label{fig:K31}}
  \caption{Two knots, $K_{G}$ and $K_{B}$, constructed by Dunfield and Gong. The knot $K_{G}$ is denoted $18_{\text{nh}00000601}$ by Dunfield and Gong.}
  \label{fig:K18andK31}
\end{figure}

Using SnapPy \cite{SnapPy} inside Sage \cite{sagemath}, Dunfield and Gong \cite{DGData,DG} were able to explicitly find a ribbon disk for $K_{B}$. This ribbon disk is given by banding together a two-component unlink, resulting in the knot $K_{B}$. We depict $K_{B}$ as a banded unlink in Figure \ref{fig:K31withband}. While SnapPy was able to find a ribbon disk for the larger knot $K_{B}$, it was unable to find any such collection of bands for the much smaller knot $K_{G}$; thus leaving the sliceness of $K_{G}$ unknown.

\begin{figure}[h]
\centering
\includegraphics[scale = 0.8]{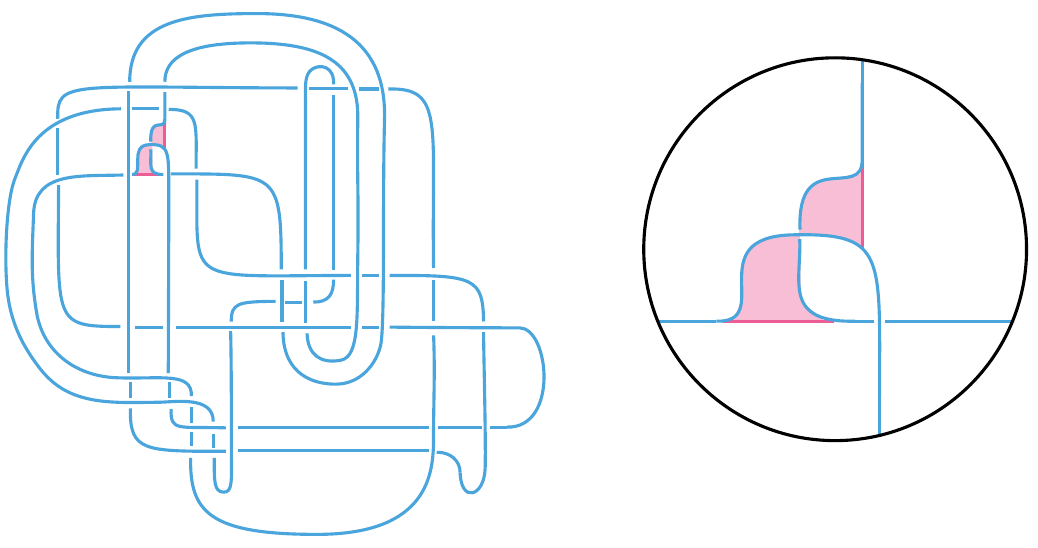}
\caption{(Left) $K_{B}$ pictured as a banded unlink with band specified in pink. (Right) We zoom in on our band used to express $K_{B}$ as a banded unlink.}
\label{fig:K31withband}
\end{figure}

Since $K_{B}$ and $K_{G}$ share the same $0$-surgery, Dunfield and Gong construct a homotopy $4$-sphere, $X_{DG}$, Manolescu and Piccirillo's strategy. Since $K_{G}$ is of unknown sliceness, $X_{DG}$ is a potential counterexample to the Smooth $4$-dimensional Poincar\'{e} Conjecture. Specifically, Theorem $5.14$ of \cite{DG} states that if $K_{G}$ is not smoohtly slice in the standard $4$-ball, then there exists an exotic $4$-sphere. However, as remarked in Section $1.17$ of \cite{DG} and Section $7.2$ of \cite{manolescu2026knotsfourmanifolds}, if $X_{DG}$ is diffeomorphic to the standard $S^{4}$ then $K_{G}$ is smoothly slice in the standard $4$-ball. Since a ribbon disk for $K_{G}$ has evaded detection, $K_{G}$ would be a potential counterexample to the Slice-Ribbon Conjecture. \newline

In this article, we standardize the homotopy $4$-sphere $X_{DG}$ constructed by Dunfield and Gong using Kirby calculus. As an immediate corollary, it follows that indeed $K_{G}$ is smoothly slice in the $4$-ball. Hence $K_{G}$ is a potential counterexample to the slice-ribbon conjecture.

Our main theorem is: 

\begin{theorem}\label{thm:mainthm}
The homotopy $4$-sphere $X_{DG} = E_{D_{K_{B}}}\cup_{\psi_{RBG}}\overline{X_{0}(K_{G})}$ is diffeomorphic to $S^4$.
\end{theorem}

\begin{corollary}\label{cor:KGslice}
    $K_{G}$ is smoothly slice in $B^{4}$.
\end{corollary}

As $K_{G}$ is smoothly slice, but not obviously ribbon, we are left with the following corollary. 

\begin{corollary}\label{cor:potentialcounterexample}
    $K_{G}$ is a potential counterexample to the Slice-Ribbon Conjecture.
\end{corollary}

Using work of Casson--Gordon, Meier--Zupan, and Miller--Zupan \cite{CG,MeZ, MiZ}, we are able to show that $K_{G}$ bounds a fibered \textit{handle-ribbon} disk in the standard $4$-ball. 

\begin{theorem}\label{thm:K18slice}
The knot $K_{G}$ bounds a fibered handle-ribbon disk in $B^{4}$.
\end{theorem}

\subsection{Organization} In Section \ref{sec:standardDGsphere}, we set the conventions of the paper; review RBG links and the Manolescu--Piccirillo strategy; and finally,  we construct a Kirby diagram for the Dunfield--Gong sphere and prove Theorem \ref{thm:mainthm}, our main theorem. In Section \ref{sec:KBslicedisk} we give the necessary preliminary material on R-links, fibered homotopy-knots, and derivative links. We finish the section by proving Theorem \ref{thm:K18slice}.\newline

\textbf{Acknowledgments.}
The author would like to give a special thank you to Nathan Dunfield for the extensive help with SnapPy. The author would also like thank Laura Starkston and Melissa Zhang for their comments on early drafts of this work, and Rob Kirby for taking interest.\ Finally, the author would like to thank his advisor, Abby Thompson, for her helpful feedback, constant support, and encouragement. The author has been supported as a GAANN Fellow through the Mathematics at UC Davis GAANN grant funded by the Department of Education grant number P200A240025.

\section{Standardizing the Dunfield--Gong Sphere}\label{sec:standardDGsphere}
We begin this section by reviewing notations and convention, in \ref{subsec:convnot}. In \ref{subsec:RBGlinkreview}, we review RBG links and the Manolescu--Piccirillo strategy. Finally, in \ref{subsec:Kdiagramandmainthm} we draw a Kirby diagram of the Dunfield--Gong $4$-sphere $X_{DG}$ and prove our main theorem, Theorem \ref{thm:mainthm}.

\subsection{Conventions and notation}\label{subsec:convnot}
Unless otherwise stated, all manifolds are smooth, compact, and oriented.\ $S^{n}$ denotes the standard smooth $n$-sphere and $B^{n}$ denotes the standard smooth $n$-ball.\ $I$ is the unit interval $[0,1]$. If $Y$ is a submanifold of $X$, we let $\eta(Y)$ denote an open regular neighborhood of $Y$ inside $X$. Given a slice disk $D$ of $K$ in $B^{4}$, we denote the \textit{exterior} of $D$ by $E_{D} = B^{4}\setminus\eta(D)$.\ Given a handle decomposition of a smooth $n$-manifold $X$, we will let $\overline{X}$ denote the result of turning the handle decomposition on $X$ upside-down. If $X$ is a manifold with boundary with a given handle decomposition, then $\overline{X}$ will correspond to a relative handle decomposition (see Section $5.5$ of \cite{GompfStip} for more details). We say that a closed $3$-manifold $Y$ is obtained by \textit{Dehn surgery} on a knot $K\subset S^{3}$ with slope $a/b$ if $Y$ can be expressed as the union of $S^{3}\setminus\eta(K)$ and a solid torus $V$ such that a meridian of $V$ goes to the $a/b$ curve on $\partial(S^{3}\setminus\eta(K))$ in preferred coordinates. In this case, we denote $Y$ by $S_{a/b}^{3}(K)$. We can adopt a similar convention for any $n$-component link $L\subset S^{3}$ where the slope $a/b$ is replaced by an $n$-tuple of boundary slopes. Related to surgery is the notion of a \textit{knot trace}. Given an integer surgery on a knot $K\subset S^{3}$, the \textit{$n$-trace of $K$}, $X_{n}(K)$, is defined as the compact $4$-manifold obtained by attaching an $n$-framed $4$-dimensional $2$-handle to $B^{4}$ along $K$. When $n=0$, we just call $X_{0}(K)$ the \textit{knot trace} of $K$. We note that $\partial(X_{n}(K))$ is $S_{n}^{3}(K)$. To distinguish between a Dehn surgery and a $4$-dimensional $2$-handle attachment, we will use bracketed coefficients to denote a surgery on a link and unbracketed coefficients to denote a handle attachment along a link. We will use the word \textit{handlebody} to mean a copy of the $3$-manifold $\natural^{n} S^{1}\times B^{2}$ for some $n$; otherwise we will specify a dimension. Given a framed link $L\subset S^{3}$ with at least two-components $L_{1}$ and $L_{2}$ with framings $n_{1}$ and $n_{2}$, a \textit{handleslide} of $L_{1}$ over $L_{2}$ is the process by which the link $L$ is replaced with a link $L' = L\setminus L_{1}\cup L_{1}'$, where $L_{1}'$ is a framed knot obtained by connecting $L_{1}$ to $L_{2}$ with a band. The framing on the component $L_{1}'$ is given by the formula $n_{1}+n_{2}\pm2\text{lk}(L_{1},L_{2})$, where $\text{lk}(L_{1},L_{2})$ denotes the algebraic linking number between $L_{1}$ and $L_{2}$. Treating a surgery on a link as a framed $4$-dimensional $2$-handle attachment, a handleslide of one component over another does not change the diffeomorphism type of the resulting $4$-manifold (See Section $5$ of \cite{GompfStip} for details). In all figures, we let a boxed integer $k$ indicate a full twist on $n$-strands; if $k>0$, then we use a positive full twist; if $k<0$, then we use a negative full twist. We also let a boxed $\pm \frac{1}{2}$ indicate a positive (resp. negative) half-twist on $n$-strands. Lastly, we diagrammatically represent a $4$-dimensional $1$-handle attachment in $B^{4}$ as a dotted circle as in \cite{Akbulut_1977}. 

\subsection{The Manolescu--Piccirillo strategy}\label{subsec:RBGlinkreview}
In order to obtain a diagram of $X_{DG}$, we must first review the general RBG construction of \cite{MP} by recalling the definition of an RBG link. 

\begin{definition}\label{def:RBGlink}
    An \textit{RBG link} $R\cup B\cup G\subset S^{3}$ is a three-component rationally framed link, with respective framings $r,b,$ and $g$ such that $H_{1}(S^{3}_{r,b,g}(R\cup B\cup G)) = \Z$, which comes equipped with homeomorphisms $\psi_{B}:S_{r,g}^{3}(R\cup G)\to S^{3}$ and $\psi_{G}:S_{r,g}^{3}(R\cup B)\to S^{3}$.
\end{definition}

\begin{remark}\label{rmk:constructlink}
Theorem $1.2$ of \cite{MP} shows that RBG links can be used to describe homeomorphisms of $3$-manifolds arising as $0$-surgeries on knots $K_{B}$ and $K_{G}$, where $K_{B}$ is the image of $B$ under the homeomorphism $\psi_{B}$ and $K_{G}$ is the image of $G$ under the homeomorphism $\psi_{G}$. In practice, $K_{B}$ (resp. $K_{G}$) is obtained by sliding $B$ (resp. $G$) over $R$ to unlink $B$ (resp. $G$) from $G$ (resp. $B$) in the link $R\cup G$ (resp. $R\cup B$). The same theorem of \cite{MP} shows that given a homeomorphism $\psi: S_{0}^{3}(K)\to S_{0}^{3}(K')$, one can construct an RBG link with $K_{B} = K'$ and $K_{G} = K$. The process is as follows. Given a homeomorphism $\psi: S_{0}^{3}(K)\to S_{0}^{3}(K')$, take the $0$-framed meridian $(\mu_{K'},0)$ of $K$. Now, let $(R,r)$ be the framed knot $\psi^{-1}(\mu_{K'},0)$ resulting from pulling the meridian of $K'$ back along our $0$-surgery homeomorphism. This gives a picture of the dual to $K'$ inside $S_{0}^{3}(K)$. Next, define the framed link $R\cup K\cup \mu_{R}\subset S^{3}$ with respective framings $r,0$ and $0$. The homeomorphism $\psi_{B}$ is given by slam-dunking $R\cup \mu_{R}$, and $\psi_{G}$ is given by pushing $R$ and $\mu_{R}$ through $\psi$, followed by a slam-dunk. 
\end{remark}

In \cite{MP}, Manolescu and Piccirillo propose the following strategy to construct exotic $S^4$'s utilizing $0$-surgery homeomorphisms: \newline

Consider pairs of knots $K$ and $K'$ with a homeomorphism $\psi$ of their $0$-surgeries. If $K$ is slice in $B^{4}$ and $K'$ is not, take a slice disk $D_{K}$ for $K$. Then remove a regular neighborhood of $D_{K}$ from $B^{4}$ to obtain $E_{D_{K}} = B^{4}\setminus\eta(D_{K})$. The boundary of this slice disk complement is precisely $S_{0}^{3}(K)$. Since $K$ and $K'$ have homeomorphic $0$-surgeries, one can then form the manifold $X_{\psi}$ by turning the $0$-trace of $K'$ upside-down and gluing it to $E_{D_{K}}$ along the $0$-surgery homeomorphism $\psi$, i.e.\ 
$$X_{\psi} = E_{D_{K}}\cup_{\psi} \overline{X_{0}(K')}.$$
\noindent One can check that $X_{\psi}$ is a homotopy $4$-sphere. By construction, $K'$ is smoothly slice in $X\setminus\eta(\{pt.\})$, but not in the standard $4$-ball. Hence, $X$ is an exotic $S^{4}$.
\subsection{A Kirby diagram of $X_{DG}$}\label{subsec:Kdiagramandmainthm}
 In \cite{DG}, Dunfield and Gong used computational methods to find the RBG link $R\cup B\cup G$ pictured in Figure \ref{fig:DGRBGlink}.

\begin{figure}[h]
    \centering
    \includegraphics[scale = 0.7]{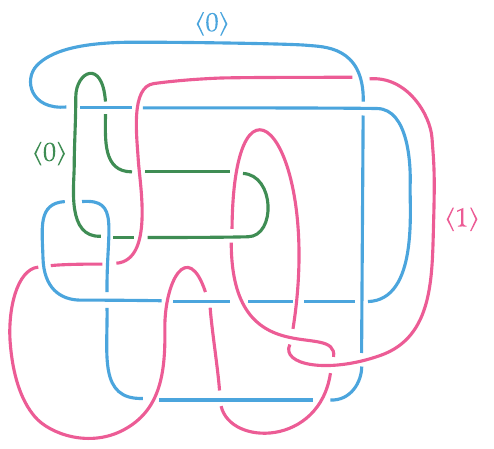}
    \caption{Dunfield and Gong's RBG link $R\cup B\cup G$ as pictured in Figure $19$ of \cite{DG} with respective framings $r = 1, b = 0$, and  $g = 0$.}
    \label{fig:DGRBGlink}
\end{figure}

Let $R\cup B\cup G$ be the RBG link given in Figure \ref{fig:DGRBGlink}. There are two associated knots $K_{B}$ and $K_{G}$, and the link $R\cup B\cup G$ defines a homeomorphism $\psi_{RBG}$ of their $0$-surgeries. The knot $K_{B}$ pictured in Figure \ref{fig:K31} is obtained from the link $R\cup B\cup G$ in Figure \ref{fig:DGRBGlink} by sliding $B$ over $R$ and canceling $R\cup G$ from the diagram. The knot $K_{G}$, pictured in Figure \ref{fig:K18}, is obtained from $R\cup B\cup G$ by sliding $G$ over $R$ and canceling $R\cup B$ from the diagram. In \cite{DG,DGData}, using SnapPy \cite{SnapPy} inside Sage \cite{sagemath}, the authors were able to find a ribbon band for $K_{B}$. This band when dualized can be used to express $K_{B}$ as a banded two-component unlink as in Figure \ref{fig:K31withband}. Using this banded unlink description of $K_{B}$, we are able to describe a ribbon disk, $D_{K_{B}}$ of $K_{B}$. Using $D_{K_{B}}$, Dunfield and Gong apply Manolescu and Picirrillo's strategy to obtain a homotopy $S^{4}$, $$X_{DG} = E_{D_{K_{B}}}\cup_{\psi_{RBG}}\overline{X_{0}(K_{G})},$$
\noindent where $E_{D_{K_{B}}} = B^{4}\setminus \eta(D_{K_{B}})$ is the ribbon-disk exterior of $D_{K_{B}}$, and $\overline{X_{0}(K_{G}})$ is the (upside-down) $0$-trace of $K_{G}$. The two manifolds are then glued along their common boundary via the homeomorphism $\psi_{RBG}$ associated to $R\cup B\cup G$. In this section, we show that $X_{DG}$ is standard.

\begin{namedtheorem}[Theorem \ref{thm:mainthm}]
The homotopy $4$-sphere $X_{DG} = E_{D_{K_{B}}}\cup_{\psi_{RBG}}\overline{X_{0}(K_{G})}$ is diffeomorphic to $S^4$.
\end{namedtheorem}

To standardize Dunfield and Gong's sphere, we need the following technical lemma to give us a Kirby diagram of $X_{DG}$. 

\begin{lemma}\label{lem:obtainKdiagram}
$X_{DG}$ is given by the Kirby diagram on the right side of Figure \ref{fig:carveddiskdiagram}. 
\end{lemma}
\begin{proof}
    Since $X_{DG}$ is the union of a $0$-trace and the ribbon-disk exterior, we can treat the gluing as occurring in steps. First attach a dual $2$-handle (corresponding to the $0$-framed $2$-handle attached along $K_{G}$) to $E_{D_{K_{B}}}$ along the image of the $0$-framed meridian of $K_{G}$ under $\psi_{RBG}$. Then, cap off with the remaining $4$-ball. This is accomplished by taking the link $R\cup B\cup G$ depicted in Figure \ref{fig:DGRBGlink}, and handlesliding $B$ over $R$ until it appears as $K_B$ as depicted in Figure \ref{fig:KBinRBG}. \newline

\begin{figure}[h]
    \centering
    \includegraphics[scale = 0.7]{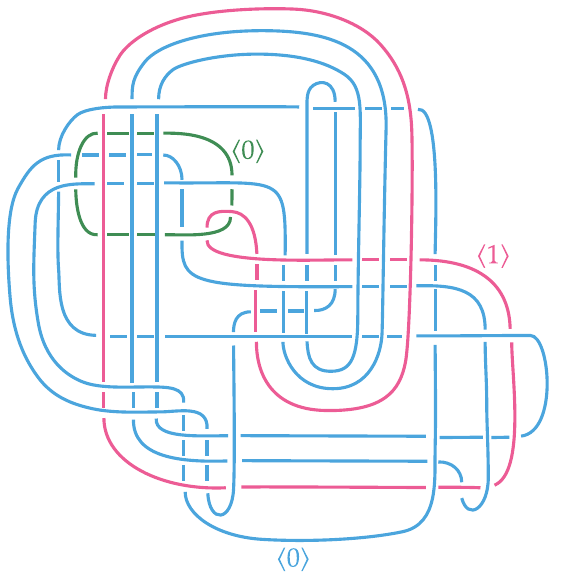}
    \caption{The link which depicts $K_{B}$ sitting inside the RBG-link after handlesliding $B$ over $R$.}
    \label{fig:KBinRBG}
\end{figure}

Since $R\cup B\cup G$ is an RBG link, performing $\langle 1,0\rangle$-surgery on $R\cup B$ must yield $S^{3}$. Thus, by handlesliding $B$ over $R$, the sublink $K_{B}\cup R$ must also yield $S^{3}$. Performing this surgery in steps, we first surger $K_{B}$ to land in $S_{0}^{3}(K_{B})$. Then, treating $R$ as a knot inside $S_{0}^{3}(K_{B})$, surgering $R$ lands us back in $S^{3}$. By taking the core of the surgery solid torus, there is a dual surgery curve $J$ to $R$ in $S^3$ with a surgery to $S_{0}^{3}(K_{B})$. Since $R$ is integer framed, the dual knot $J$ may be represented as a $0$-framed meridian to $R$. However, in the link $K_{B}\cup R\cup G$, we see that $G$ is drawn as a $0$-framed meridian to $R$. Thus, we can take $J$ to be the image of $G$ after performing our $\langle 1,0\rangle$-surgery on $R\cup K_{B}$. However, since $R\cup K_{B}$ is handleslide equivalent to $R\cup B$, this has the same effect on $G$ as the homeomorphism $\psi_{G}$ in our RBG link. Hence $J$ is isotopic to $K_{B}$ in $S^{3}$. Thus, the surgery dual to $K_{G}$ in $S_{0}^{3}(K_{B})$ must be $R$ with $1$-framing. Hence, we can take the image of the $0$-framed meridian to $K_{G}$ under the associated RBG homeomorphism to be $R$ with $1$-framing.

Now by keeping track of the band we surgered to obtain $D_{K_{G}}$, we can obtain a Kirby diagram for $X_{DG}$ from the link $K_{B}\cup R$ as follows: Dot $K_{B}$ as in \cite{AK} (or Section $6.2$ \cite{GompfStip}), attach a $1$-framed $2$-handle along $R$, and finish by adding a $4$-handle. The result is depicted on the left side of Figure \ref{fig:carveddiskdiagram}. By performing a ribbon move on the dotted component corresponding to our twisted band (see Section $1.4$ of \cite{akb} or Section $6.2$ of \cite{GompfStip} for details), we obtain the diagram on the right side of Figure \ref{fig:carveddiskdiagram}. This gives us a Kirby diagram of $X_{DG}$.

\begin{figure}[h]
    \centering
    \includegraphics[scale = 0.7]{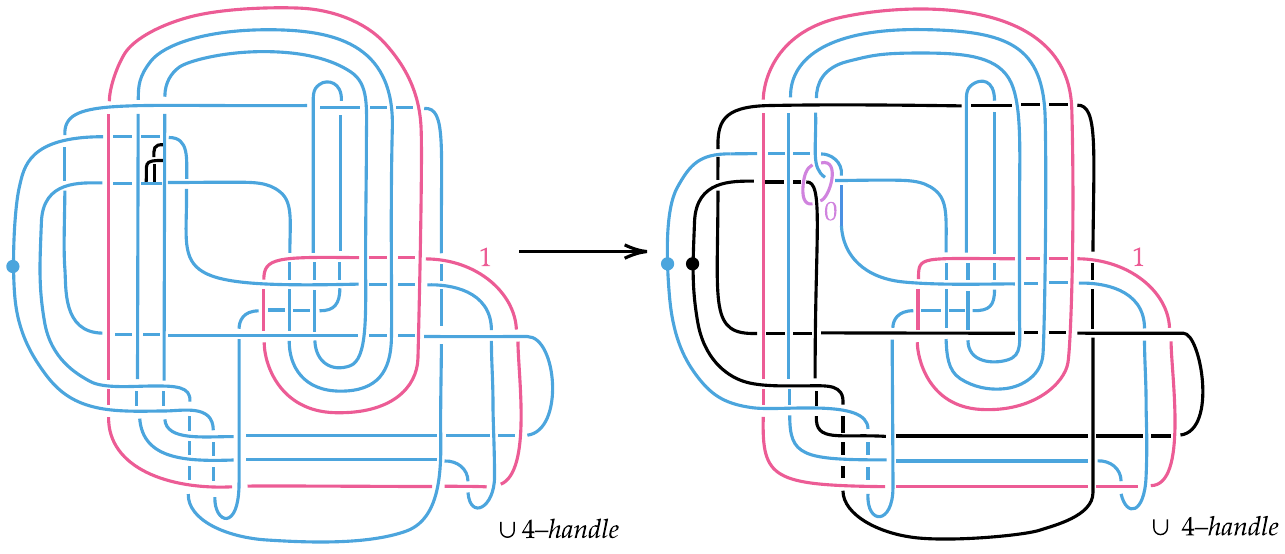}
    \caption{(Left) A Kirby diagram of $X_{DG}$ obtained from the link $K_{B}\cup R$ by dotting $K_{B}$, adding a $1$-framed $2$-handle to $R$, then capping off with a $4$-ball. (Right) A Kirby diagram for $X_{DG}$ obtained by performing a ribbon move on the dotted component in the figure on the left.}
    \label{fig:carveddiskdiagram}
\end{figure}


\end{proof}

 With a Kirby diagram of $X_{DG}$, we can now prove Theorem \ref{thm:mainthm}.

\begin{proof}[Proof of Theorem \ref{thm:mainthm}]
We proceed using standard Kirby calculus. From Lemma \ref{lem:obtainKdiagram}, we represent $X_{DG}$ by the Kirby diagram on the right side of Figure \ref{fig:carveddiskdiagram}. Taking our red $1$-framed component $R$, we slide $R$ under the blue $1$-handle using the untwisted band specified by the orange arrow on the left side of Figure \ref{fig:homotopy4ballwithslide}. The result of this slide is given by surgering along the specified band, depicted on the right side of Figure \ref{fig:homotopy4ballwithslide}.

\begin{figure}[h]
    \centering
    \includegraphics[scale = 0.7]{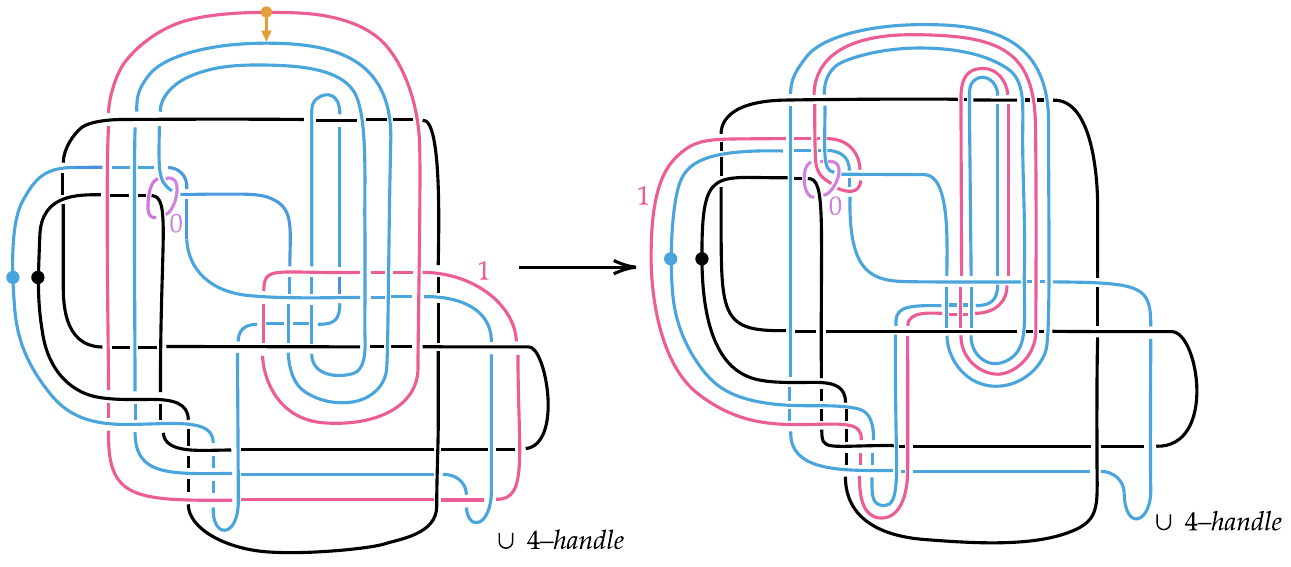}
    \caption{(Left) We slide our red $1$-framed component under the blue $1$-handle using an untwisted band specified by the orange arrow. (Right) The result of the previous handleslide yields a presentation of $X_{DG}$ with a canceling $(1,2)$-Hopf pair.}
    \label{fig:homotopy4ballwithslide}
\end{figure}

In a slight abuse of notation, we rename the red component resulting from our slide $R$. One can check that $R$ is unlinked with the blue $1$-handle and forms a canceling $(1,2)$-Hopf pair with the black $1$-handle. Canceling this Hopf pair from the diagram, we conclude that $X_{DG}$ can be built from a single $0$-handle, $1$-handle, $2$-handle, and $4$-handle. To simplify handle calculus, we appeal to Gabai's resolution of the Property R conjecture in \cite{Gab}. We take $B'$ to be the union of the $0,1$, and $2$-handles in the diagram on the right side of Figure \ref{fig:homotopy4ballwithslide}. Since we are able to cap off $B'$ with a $4$-handle to obtain a homotopy $S^{4}$, $B'$ must be a homotopy $4$-ball built from a single $0$-handle, $1$-handle, and $2$-handle. The boundary of $B'$ is described as $\langle 0,n\rangle$-surgery on a two-component link where the $0$-framed component is taken to be the unknot. By performing surgery on the unknotted component first, we can take the second component of our link as a knot in $S^1\times S^2$. By surgering the second component, we land back in $S^3$. From Gabai's proof of the Property R conjecture \cite{Gab}, we can isotope this second component in $S^1\times S^2$ to appear as a copy of $S^1\times\{pt.\}$. This isotopy can be recognized diagrammatically by sliding the attaching sphere of our $2$-handle under our dotted $1$-handle in a manner similar to Lemma $2.4$ of \cite{MiZ}. The result of these slides is a canceling $(1,2)$-Hopf pair. Canceling this pair from the diagram, we obtain $B^4$.
\end{proof}

To see that $K_{G}$ is slice in the standard $4$-ball, we recall the Trace Embedding Lemma (for a proof, we refer the reader to Lemma $1.3$ of \cite{Picc2020}).

\begin{lemma}[Trace Embedding Lemma]\label{lem:traceemblem}
A knot $K\subset S^3$ is smoothly slice in $B^{4}$ if and only if $X_{0}(K)$ (equivalently $\overline{X_{0}(K)}$) smoothly embeds into $S^{4}$.
\end{lemma}
This immediately implies:

\begin{namedtheorem}[Corollary \ref{cor:KGslice}]
The knot $K_{G}$ bounds a smoothly slice disk $D_{K_{G}}$ in the standard $4$-ball.
\end{namedtheorem}

As $K_{G}$ is not known to be ribbon, 

\begin{namedtheorem}[Corollary \ref{cor:potentialcounterexample}]
    $K_{G}$ is a potential counterexample to the Slice-Ribbon Conjecture.
\end{namedtheorem}

In the next section, we further analyze the slice disk $D_{K_{G}}$ obtained from Corllary \ref{cor:KGslice}.

\section{$K_{B}$ bounds a fibered handle-ribbon disk}\label{sec:KBslicedisk}
In this section we show that Dunfield and Gong's knot $K_{G}$ bounds a fibered handle-ribbon disk in $B^{4}$. We begin by recalling the following definition due to Casson and Gordon in \cite{CG}.

\subsection{Fibered, homotopy ribbon knots}
\begin{definition}\label{def:homrib}
Let $V$ be a homotopy $4$-ball, and let $K\subset S^{3} = \partial V$ be a knot. We say $K$ is \textit{homotopy ribbon} in $V$ if there is a smooth properly embedded disk $D$ in $V$ such that $\partial D = K$ and the induced map $i_{*}: \pi_{1}(S^{3}\setminus K)\to \pi_{1}(V\setminus D)$ is surjective. We call such a disk a \textit{homotopy ribbon disk in} $V$.
\end{definition}

In addition to homotopy ribbonness, we also have the notion of \textit{handle-ribbonness} as in \cite{MiZ} (also known as \textit{strong homotopy ribbonness} in other literature \cite{MilZemp,HKP}).

\begin{definition}\label{def:hanrib}
Let $V$ be a homotopy $4$-ball and let $K\subset S^{3} = \partial V$ be a knot. We say $K$ is \textit{handle-ribbon in} $V$ if there is a smooth properly embedded disk $D$ in $V$ such that $\partial D = K$ and $V\setminus\eta(D)$ admits a handle decomposition without any $4$-dimensional $3$-handles. We call such a disk $D$ a \textit{handle-ribbon disk} in $V$. 
\end{definition}

The notions of homotopy ribbonness and handle-ribbonness serve as alternative notions of ribbonness in the setting of more general homotopy $4$-balls, where a canonical Morse function is not obvious. However, when we take our ambient homotopy $4$-ball to be the standard $B^{4}$, we can stratify these notions of sliceness for knots into the following set containments:
$$\{\text{ribbon knots}\}\subset\{\text{handle-ribbon knots}\}\subset\{\text{homotopy ribbon knots}\}\subset\{\text{slice knots}\}.$$
The Slice-Ribbon Conjecture implies that all the above containments are equalities. 

The next notion of importance to us is \textit{fiberedness}. If $X$ is a compact connected smooth manifold, and $\vphi:X\to X$ is some diffeomorphism of $X$, we can form the \textit{mapping torus} $$X\times_{\vphi}S^{1} = (X\times I)/\sim,$$
where $I = [0,1]$ and $\sim$ is an equivalence relation defined by $(x,0)\sim (\vphi(x),1)$ for all $x\in X$. The map $\vphi$ is called the \textit{monodromy} of the mapping torus and a copy of $X\times\{\theta\}$ for $\theta\in I$ is called a \textit{fiber} of the mapping torus. A fibered knot is as follows:

\begin{definition}\label{def:fib}
A knot $K\subset S^{3}$ is \textit{fibered} if $S^{3}\setminus\eta(K)$ admits a mapping torus structure $F\times_{\vphi} S^{1}$, where $F$ is a minimal genus Seifert surface for $K$, and $\vphi:F\to F$ is a diffeomorphism of $F$ which restricts to the identity on $\partial F = K$. The map $\vphi$ is called the \textit{monodromy} of $K$. For $\theta\in I$, a copy of $F\times\{\theta\}\subset S^{3}\setminus\eta(K)$ is called a \textit{fiber} of $K$.
\end{definition}

Performing $0$-surgery on a fibered knot $K$ in $S^3$ yields a closed fibered $3$-manifold, $S_{0}^{3}(K)$, with closed fiber $\hat{F} = F\cup D^{2}$ obtained by capping off each fiber of $K$ with a meridinal disk of the surgery solid torus. The monodromy of $S_{0}^{3}(K)$ can be taken to be the map $\hat{\vphi}:\hat{F}\to\hat{F}$ defined by $$\hat{\vphi} = \begin{cases}
\vphi,\ x\in F\\
\text{id},\ x\in D^{2}\ .
\end{cases}$$
\noindent We call $\hat{\vphi}$ the \textit{closed monodromy of} $K$. 

We introduce the notion of a \textit{monodromy extension over handlebodies} as in Section $5$ of \cite{CG}:

\begin{definition}\label{monoext}
  Let $K\subset S^{3}$ be a fibered knot with monodromy $\vphi$, $0$-surgery $S_{0}^{3}(K)$, closed fiber $\hat{F}$, and closed monodromy $\hat{\vphi}$. We say that $\hat{\vphi}$ \textit{extends over a handlebody} if there is an abstract handlebody $H$ such that $\partial H = \hat{F}$, and a diffeomorphism $\Phi: H\to H$ such that $\Phi|_{\partial H} = \hat{\vphi}$. 
\end{definition}

In addition to fiberedness of knots, we can similarly introduce the notion of a fibered disk: 

\begin{definition}\label{def:fibdisk}
Let $D\subset V$ be a smooth properly embedded disk in a homotopy $4$-ball $V$. We say $D$ is \textit{fibered} if $V\setminus\eta(D)$ admits the structure of a mapping torus $H\times_{\Phi}S^1$ for some compact $3$-manifold $H$ and some diffeomorphism $\Phi:H\to H$ such that $\Phi|_{\partial \eta(D)}$ acts as projection onto the second factor. 
\end{definition}

Casson and Gordon characterized fibered homotopy ribbon knots in terms of monodromy extensions over handlebodies. We recount this theorem:

\begin{theorem}[Theorem $5.1$ \cite{CG}]\label{thm:fibhomrib}
A fibered knot $K\subset S^{3}$ is homotopy ribbon in \textit{some} homotopy $4$-ball if and only if the closed monodromy of $K$ extends over handlebodies.
\end{theorem}

The Casson--Gordon theorem also gives the following corollary:

\begin{corollary}[Corollary $5.4$ \cite{CG}]\label{cor:fibdisk}
Suppose $K$ is a fibered knot in $S^3$ with closed monodromy $\vphi$. If $K$ is homotopy ribbon in some homotopy $4$-ball $V$, then there is an extension $\Phi$ of $\hat{\vphi}$, a homotopy $4$-ball $V'$, and a fibered homotopy ribbon disk $D_{\Phi}$ in $V'$ such that $V'\setminus\eta(D_{\Phi})$ fibers over handlebodies with monodromy $\Phi$.
\end{corollary}

\begin{remark}
The disk $D_{\Phi}$ in Corollary \ref{cor:fibdisk} can actually be taken to be handle-ribbon in $V'$.
\end{remark}

If we consider the case where $K$ is a fibered homotopy ribbon knot in $B^{4}$, then Corollary \ref{cor:fibdisk} implies the existence of a homotopy $4$-ball $V'$ such that $K$ possess a fibered homotopy ribbon disk in $V'$. This leads to two natural questions, noted by Miller in \cite{Miller}: 

\begin{question}[Question $1.3$ of \cite{Miller}]\label{q:diffeoB4}
Is $V'$ diffeomorphic to $B^{4}$?
\end{question}

\begin{question}[Question $7.1$ of \cite{Miller}]\label{q:fibdiskB4}
Let $K$ be a fibered homotopy ribbon knot in $S^{3}$. Does $K$ bound a fibered homotopy ribbon disk in $B^{4}$?
\end{question}

We note that if the answer to either of these questions is ``no," then there exists an exotic $4$-ball.  \newline

To aid in the study of fibered homotopy ribbon knots and monodromy extensions, we recall the notion of \textit{R-links} and \textit{derivatives} (see \cite{GST,MeZ,MiZ} for details).

\begin{definition}
Let $L\subset S^{3}$ be an n-component link. We say $L$ is an \textit{R-link} if $0$-framed surgery on $L$ gives $\#^{n}(S^1\times S^2)$.
\end{definition}

We emphasize that it is crucial in the definition of an  R-link that the number of components is equal to the number of summands of $S^1\times S^2$. If $L$ is an $n$-component R-link, then $L$ gives rise to a closed $4$-manifold $X$ built from a single $4$-dimensional $0$-handle, $n$ $4$-dimensional $2$-handles, $n$ $4$-dimensional $3$-handles, and a single $4$-handle. As $X$ was built with no $1$-handles, $X$ is simply-connected. In addition, a Mayer--Vietoris argument shows that the $3$-handles homologically cancel the $2$-handles, leaving us with a homotopy $S^4$. By work of Freedman \cite{Fr}, we conclude that $X$ is \textit{homeomorphic} to $S^4$, but the two are not obviously diffeomorphic. \newline

Given an R-link $L$ in $S^3$, there are two types of operations we can perform on $L$ which do not change the diffeomorphism type of the resulting homotopy $4$-sphere $X$. The first move allows us to potentially change the isotopy class of a framed component of $L$ inside $S^3$. This is accomplished by treating our surgery on $L$ as a $4$-dimensional $2$-handle attachment and performing \textit{handleslides} on the components of $L$. This move, in turn, gives rise to the Generalized Property R Conjecture (Problem $1.82$ of \cite{kirby-web})):

\begin{conjecture}[Generalized Property R Conjecture]
Every $n$-component R-link is handleslide equivalent to an $n$-component $0$-framed unlink.
\end{conjecture}

\noindent The significance of the Generalized Property R Conjecture is that it implies any homotopy $4$-sphere associated to an R-link can be standardized without adding canceling handle-pairs. \newline

The second operation allows us to potentially change the number of components of $L$ inside $S^3$. If $\mathcal{U}\subset S^3$ is an $m$ component unlink, then $0$-surgery on $\mathcal{U}$ yields $\#^{m}(S^1\times S^2)$. If we are given an R-link $L\subset S^3$, then performing $0$-framed surgery on $L\sqcup\mathcal{U}$ gives us $\#^{n+m}(S^1\times S^2)$; hence, $L\sqcup\mathcal{U}$ is an R-link as well. Letting $X$ be the smooth homotopy $4$-sphere defined by our R-link $L$, and $X'$ the homotopy $4$-sphere defined by $L\sqcup\mathcal{U}$, there are $m$ canceling $2,3$-handle pairs coming from $\mathcal{U}$. This gives a diffeomorphism from $X'$ to $X$. This discussion leads us to our next definition.

\begin{definition}\label{def:stableequiv}
Two R-links $L$ and $L'$ in $S^3$ are \textit{stably equivalent} if there are two unlinks $\mathcal{U}$ and $\mathcal{U'}$ of unknotted components in $\Sigma$ such that $L\sqcup\mathcal{U}$ is handleslide equivalent to $L'\sqcup\mathcal{U'}$.
\end{definition}

R-links also help in the study of sliceness through the following relationship, made explicit in \cite{MiZ}.

\begin{lemma}[Lemma $3.1$ of \cite{MiZ}]
A knot $K\subset S^3$ is handle-ribbon in some homotopy $4$-ball if and only if $K$ is a component of some R-link $L$. The core of the $2$-handle attached to $K$ can be taken to serve as a handle-ribbon disk for $K$.
\end{lemma}

Another notion of importance to us will be \textit{derivative links}. 

\begin{definition}\label{def: derivative}
 For a knot $K\subset S^3$ and an oriented genus $g$ Seifert surface $F$ for $K$, a \textit{derivative} for $K$ in $F$ is a $g$-component link $L\subset F$ such that $F\setminus L$ is a connected planar surface, and $\text{lk}(L_{i},L_{j}^{+}) = 0$ for all $1\leq i,j,\leq g$, where $L_{j}^{+}$ is a parallel push-off of $L_{j}$ in the positive direction defined by $F$. If in addition to $L$ being a derivative to $K$, $L$ is also an R-link, we call $L$ an \textit{R-link derivative} of $K$.
\end{definition}

Using derivative links, we can give a characterization of monodromy extensions in the sense of Casson and Gordon \cite{CG}. Let $K$ be a knot in $S^3$ with Seifert surface $F$ and a derivative link $L$ for $F$. Use $L$ to define an abstract handlebody $H$ as follows:

Performing $0$-surgery on $K$, capping off $F$ to obtain a closed orientable surface $\hat{F}$. Carrying $L$ into the $0$-surgery. Define an abstract handlebody $H$ with boundary precisely $\hat{F}$, by attaching $3$-dimensional $2$-handles along $L$ and capping off with a $3$-handle. $H$ is a handlebody since $L$ cuts $F$ into a planar surface and has zero pairwise algebraic linking with respect to the surface framing. 

\begin{definition}
Let $K$ be a fibered knot with fiber $F$ whose closed monodromy $\hat{\vphi}$ extends over a handlebody, and let $L\subset F$ be a derivative of $F$.\ $L$ is a \textit{Casson--Gordon derivative} (or \textit{CG-derivative}) if the closed monodromy $\hat{\vphi}$ admits an extension over the handlebody $H$ defined by $L$.
\end{definition}

To highlight the connection between R-links and CG-derivatives, we recall Theorem $3.3$ of \cite{MiZ} (alternatively, Theorem $1.5$ of \cite{MeZ}).

\begin{theorem}[Theorem $3.3$ of \cite{MiZ}]\label{thm:RlinkCGderiv}
If $K\cup J\subset S^{3}$ is a two-component R-link such that $K$ and $J$ are knots and $K$ is fibered, then $K\cup J$ is stably equivalent to a link $K\cup L^{+}$, where $L^{+}$ is a CG-derivative of $K$. In this case, the closed monodromy of $K$ extends over the handlebody determined by $L^{+}$. 
\end{theorem}
Using Theorem \ref{thm:RlinkCGderiv} and a method of Nakamura in \cite{Nak}, we will show that $K_{G}$ bounds a fibered, handle-ribbon disk in $B^{4}$.
\subsection{Constructing a fibered handle-ribbon disk for $K_{G}$}
We now show that the slice disk $D_{K_{G}}$ for $K_{G}$ obtained from Corollary \ref{cor:KGslice} is a fibered handle-ribbon disk.
\begin{namedtheorem}[Theorem \ref{thm:K18slice}]
The knot $K_{G}$ bounds a fibered handle-ribbon disk in $B^{4}$.
\end{namedtheorem}
\begin{proof}
    \label{fig:upsidedowndiskextwithRBG}
    \label{fig:upsidedowndiskextwithRBG}
To show that $K_{G}$ bounds a fibered handle-ribbon disk in $B^{4}$, we use the method outlined in Section $4.1$ of \cite{Nak} to turn the handle decomposition obtained from Lemma \ref{lem:obtainKdiagram} upside-down. Representing $K_{B}$ as the banded unlink in Figure \ref{fig:K31withband}, we obtain the ribbon disk $D_{K_{B}}$. We wish to present $\overline{E_{D_{K_{B}}}}$ as a relative handle decomposition. First, add a bracketed $\langle0\rangle$-framing to $K_{B}$. Next, place in a small $0$-framed unknot $U$ that links with our ribbon band in the sense of \cite{Thomp} (colored black in Figure \ref{fig:upsidedowndiskextwithRBG}). To finish, cap off the resulting manifold with two $4$-dimensional $3$-handles and a $4$-handle. The result is a copy of $\overline{E_{D_{K_{B}}}}$, expressed as a relative Kirby diagram on the left side of Figure \ref{fig:upsidedowndiskextwithRBG}.

\begin{figure}[h!]
    \centering
    \includegraphics[scale = 0.69]{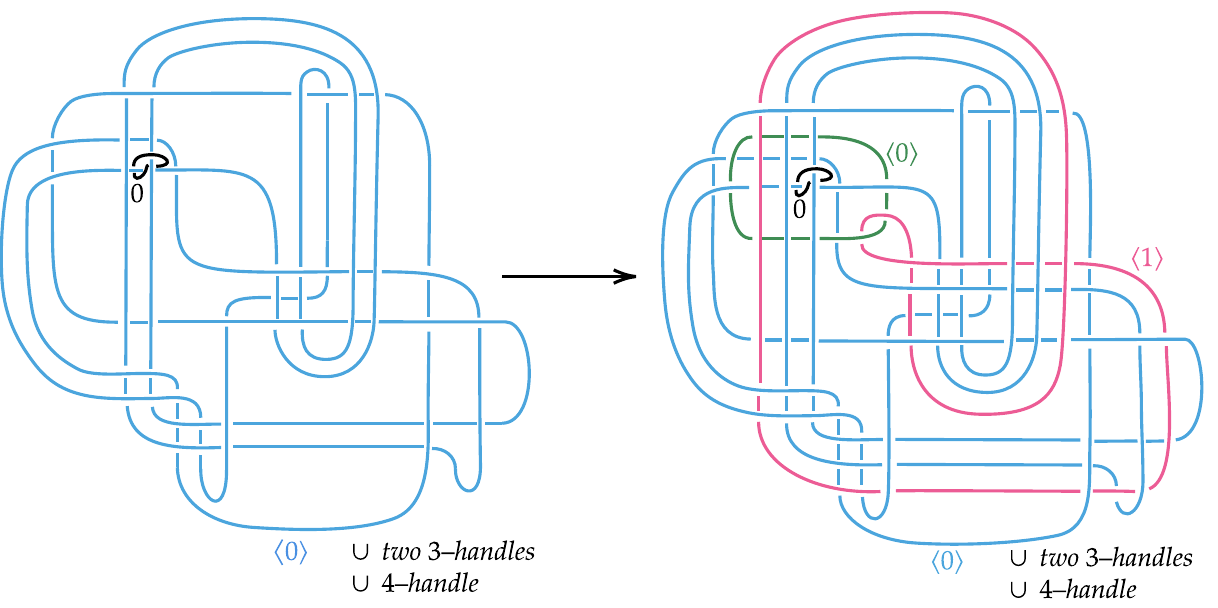}
    \caption{(Left) A relative Kirby diagram for $\overline{E_{D_{K_{B}}}}$. (Right) A different presentation $E_{D_{K_{B}}}$ as a relative Kirby diagram of a with the bracketed-framed link $R\cup K_{B}\cup U\cup G$.
crossings with R.}
    \label{fig:upsidedowndiskextwithRBG}
\end{figure}

To obtain a Kirby diagram of the manifold $X = X_{0}(K_{G})\cup_{\psi_{RBG}}\overline{E_{D_{K_{B}}}}$, we pull the relative handle diagram for $\overline{E_{D_{K_{B}}}}$ back along the homeomorphism $\psi_{RBG}$ associated to the RBG link $R\cup B\cup G$ from Figure \ref{fig:DGRBGlink}. In doing so, we make sure to keep track of the image of the sublink $G\cup U$ under this homeomorphism. First, insert a disjoint $\langle 0,1\rangle$-framed Hopf link $G\cup R$, with components colored green and red, respectively. By sliding $K_{B}$ over $G$, we change crossings between $K_{B}$ and $R$ until we obtain the link $R\cup K_{B}\cup U\cup G$ pictured on the right side of Figure \ref{fig:upsidedowndiskextwithRBG}. We can obtain the two-component link $L = K_{G}\cup \psi_{RBG}(U)$ depicted on the left side of Figure \ref{fig:exteriorwithrlink} by sliding $K_{B}$ over $R$ using the handlesides given in Figure \ref{fig:handleslidingK31inRBGlink}, followed by deleting $R\cup B$ after sliding $K_{G}\cup \psi_{RBG}(U)$ over $B$ to become disjoint from $R$. In a slight abuse of notation, we rename $\psi_{RBG}(U)$ as $U$.

\begin{figure}[h!]
    \centering
    \includegraphics[scale = 0.68]{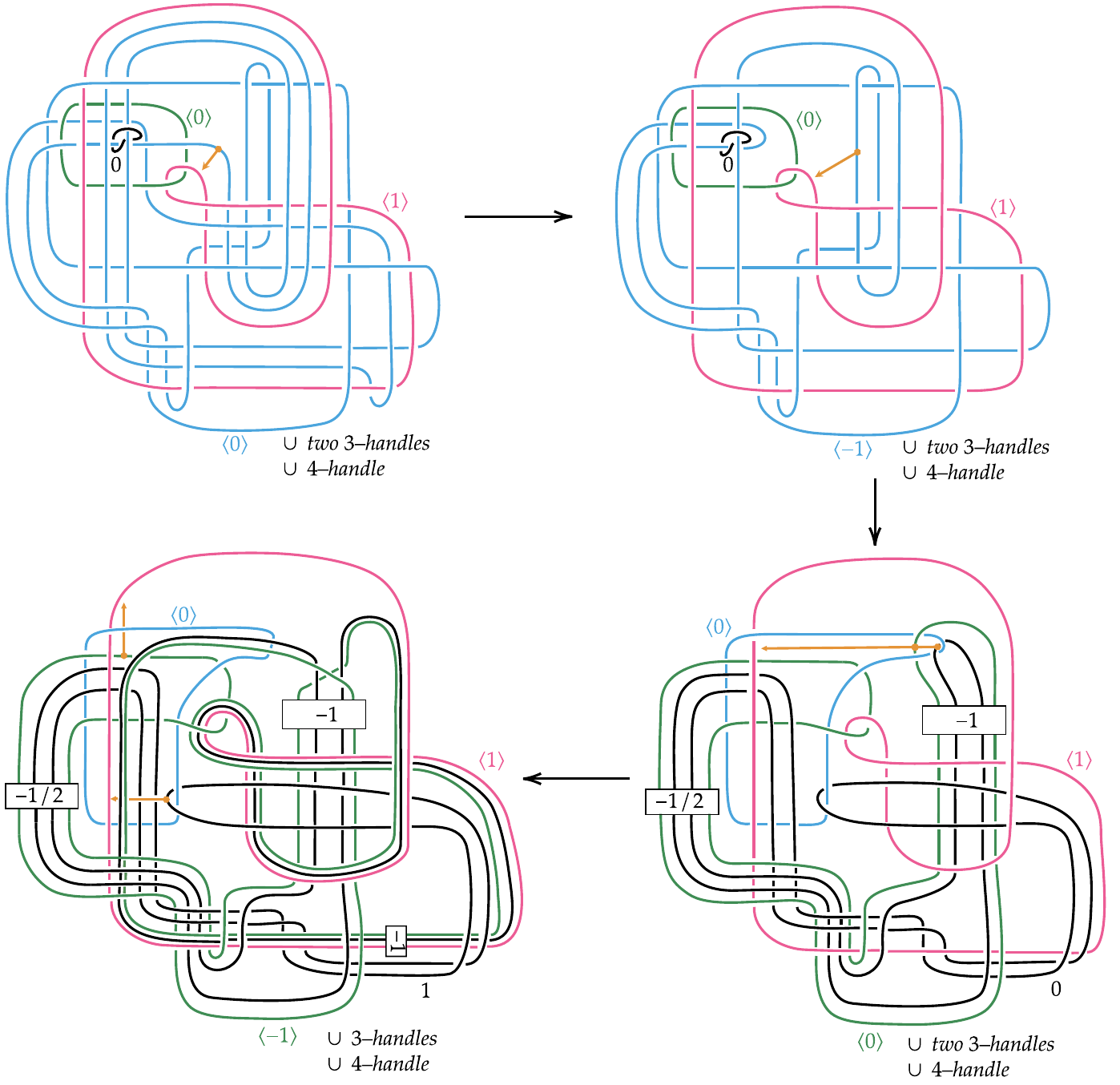}
    \caption{(Top Row) The link $R\cup K_{B}\cup U\cup G$ is transformed into the link $R\cup B\cup U\cup G$ using the handleslides depicted by the orange arrows. (Bottom Row) The link $R\cup B\cup U\cup G$ is transformed into the link $K_{G}\cup U$ pictured on the left side of Figure \ref{fig:exteriorwithrlink} by performing the handleslides of $G$ and $U$ over $R$ depicted by the orange arrows.}
    \label{fig:handleslidingK31inRBGlink}
\end{figure}

\begin{figure}[h]
    \centering
    \includegraphics[scale = 0.7]{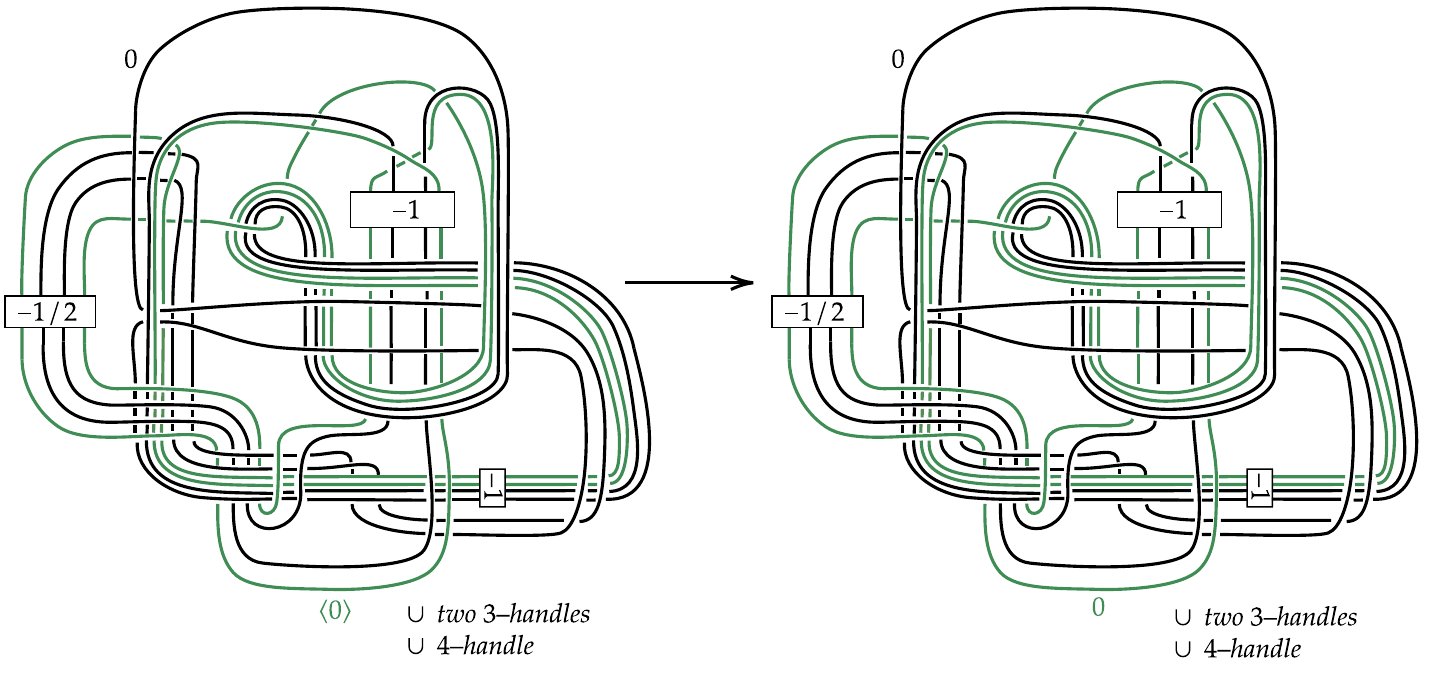}
    \caption{(Left) A relative Kirby diagram for $\overline{E_{D_{K_{B}}}}$ expressed as the link $K_{G}\cup U$ obtained from the handleslides in Figure \ref{fig:handleslidingK31inRBGlink}. (Right) The result of plugging $X_{0}(K_{G})$ into the diagram of $\overline{E_{D_{K_{B}}}}$ to obtain a two-component R-link.}
    \label{fig:exteriorwithrlink}
\end{figure}

Following \cite{Nak}, we plug the trace $X_{0}(K_{G})$ into our diagram of $\overline{E_{D_{K_{B}}}}$ by unbracketing $K_{G}$ to obtain a Kirby diagram for $X$ given on the right side of Figure \ref{fig:exteriorwithrlink}. Since attaching $0$-framed $2$-handles along $L$ results in a $3$-manifold that can be capped off by a genus two $4$-dimensional $1$-handlebody, $L$ must be a two-component R-link. As $X$ was obtained by turning the handle decomposition of $X_{DG}$ obtained from Lemma \ref{lem:obtainKdiagram} upside-down, Theorem \ref{thm:mainthm} guarantees that $X$ is diffeomorphic to $S^{4}$. Thus, there is a natural trace embedding of $X_{0}(K_{G})$ into $S^4$ determined by the R-link $L$. From Theorem \ref{thm:RlinkCGderiv}, the core of the $2$-handle attached to $K_{G}$ constitutes a handle-ribbon disk $D_{K_{G}}$ in $B^{4}$. We note that $D_{K_{G}}$ is the same disk obtained from Corollary \ref{cor:KGslice}. Using the \texttt{knot\_floer\_homology()} package inside SnapPy \cite{SnapPy}, one can verify that $K_{G}$ is a genus five fibered knot. Letting $\hat{F}$ denote the capped off fiber of $K_{G}$ in $S_{0}^{3}(K_{G})$, and $\hat{\vphi}:\hat{F}\to\hat{F}$ denote the closed monodromy of $K_{G}$, we know from Theorem \ref{thm:fibhomrib} that $\hat{\vphi}$ must extend over a handlebody. We use the R-link $L$ to describe a monodromy extension. By Theorem \ref{thm:RlinkCGderiv}, $L$ is stably equivalent to a six-component link $L' = K_{G}\cup L^{+}$, where $L^{+}$ is a CG-derivative of $K$. Since $L'$ is stably equivalent to $L$, $L'$ defines a homotopy $4$-sphere which is diffeomorphic to $S^4$. In addition, attaching $2$-handles along $L^{+}$ with the framing induced by the fiber of $K_{G}$ in $S_{0}^{3}(K_{G})$ and capping off with a genus six $4$-dimensional $1$-handlebody will define the same slice disk exterior as the diagram on the left side of Figure \ref{fig:exteriorwithrlink}. Since $L^{+}$ was a CG-derivative of $K$, Theorem \ref{thm:RlinkCGderiv} gives us an abstract handlebody $H$ with $\partial H = \hat{F}$ defined by $L^{+}$ over which the closed monodromy $\hat{\vphi}$ must extend. We call this monodromy extension $\Phi$. Using $\Phi$, we conclude that $B^{4}\setminus\eta(D_{K_{G}})$ must be the handlebody bundle $H\times_{\Phi}S^{1}$, the desired result.
\end{proof}

\begin{remark}
    We can see that the above theorem answers Questions \ref{q:diffeoB4} and \ref{q:fibdiskB4} in the affirmative for the monodromy extension $\Phi$ defined by the CG-derivative $L^{+}$.
\end{remark}

Although we have exhibited a fibered handle-ribbon disk $D_{K_{G}}$ for $K_{G}$ through a trace embedding, we were unable to show that $D_{K_{G}}$ is ribbon. This still allows for the possibility for $K_{G}$ to serve as a potential counterexample to the Slice-Ribbon Conjecture. We end with the following two questions.

\begin{question}
    Is the fibered handle-ribbon disk $D_{K_{G}}$ of $K_{G}$ from Theorem \ref{thm:K18slice} a ribbon disk for $K_{G}$?
\end{question}
If not, we can more generally ask: 
\begin{question}
    Is $K_{G}$ ribbon?
\end{question}

\bibliographystyle{alpha}
\bibliography{references.bib}
\end{document}